\documentclass{amsart}

\usepackage{amsmath}
\usepackage{amsfonts}
\usepackage{amssymb}

\theoremstyle{definition}
\newtheorem{theorem}{Theorem}[section]
\newtheorem{proposition}[theorem]{Proposition}
\newtheorem{lemma}[theorem]{Lemma}

\theoremstyle{definition}

\theoremstyle{remark}
\newtheorem{remark}[theorem]{Remark}

\newcommand{\ar}{r_1 - \lambda}
\newcommand{\br}{r_2 - \lambda}
\allowdisplaybreaks[1]

\numberwithin{equation}{section}

\begin{document}

\title[]{Spectrum and fine spectrum of generalised lower triangular triple band matrices over the sequence space $l_p$}


\author{Arnab Patra}
\address{Arnab Patra, Department of mathematics, Indian Institute of Technology Kharagpur, India 721302}
\curraddr{}
\email{arnptr91@gmail.com}
\thanks{}

\author{P. D. Srivastava}
\address{P. D. Srivastava, Department of mathematics, Indian Institute of Technology Kharagpur, India 721302}
\curraddr{}
\email{pds@maths.iitkgp.ernet.in}
\thanks{}

%
%
%

\begin{abstract} 
The spectrum of triangular band matrices defined on the sequence spaces where the entries of each band is a constant or convergent sequence is well studied. In this article, the spectrum and fine spectrum of a new generalised difference operator defined by a lower triangular triple band matrix on the sequence space $l_p (1 \leq p < \infty)$ are obtained where the bands are considered as periodic sequences. The approximate point spectrum, defect spectrum, compression spectrum and the Goldberg classification of the spectrum are also discussed. Suitable examples are given in order to supplement the results. Several special cases of our findings are discussed which confirm that our study is more general and extensive. 
\end{abstract}

\maketitle
\section{Introduction}

In infinite dimensional Banach spaces, the spectrum of a bounded linear operator has more complex structure than the finite dimensional case. It can be partitioned into three disjoint parts, namely point spectrum, continuous spectrum and residual spectrum and these three parts of the spectrum is termed as fine spectrum of the operator.

Several researchers have studied the spectrum and fine spectrum of various bounded linear operators defined on various sequence spaces. Altay and Ba{\c{s}}ar \cite{deltac0c}, Kayaduman and Furkan \cite{deltal1bv}, Akhmedov and Ba{\c{s}}ar \cite{deltabvp} obtained the spctrum and fine spectrum of the difference operator $\Delta$ over the sequence spaces $c_0,$ $c$ and $l_1,$ $bv$ and $bv_p (1 \leq p < \infty)$ respectively. The fine spectrum of the generalised difference operator $B(r, s)$ is studied by Altay and Ba{\c{s}}ar \cite{c0c}, Furkan et al. \cite{l1bv}, Bilgi{\c{c}} and Furkan \cite{lpbvp} over the sequence spaces $c_0$ and $c$, $l_1$ and $bv$, $l_p$ and $bv_p$ $(1<p< \infty)$ respectively. Furkan et al. \cite{c0c3,lpbvp3}, Bilgi{\c{c}} and Furkan 
\cite{l1bv3} further generalised these results to the operator $B(r, s, t).$ The fine spectrum of triangular Toeplitz operator defined on $c_0$ and $c$ is obtained by Altun \cite{altun}. Srivastava and Kumar \cite{v_l1,uv_l1} have studied the spectrum and fine spectrm of the generalised differnece operators $\vartriangle_v$ and $\vartriangle_{uv}$ on $l_1.$  Later on, El-Shabrawy \cite{ab_lp,abc0} has studied the fine spectrum of the operator $\vartriangle_{ab}$ on the sequence space $l_p$ $(1< p< \infty)$ and $c_0$ respectively. The spectrum of the upper triangular matrix $A(\tilde{r}, \tilde{s})$ defined on $l_1$ and $l_p$ $(1< p< \infty)$ is studied by Karaisa \cite{upper_lp, upperl1}. In 2012, the fine spectrum of lower triangular triple band matrix $\Delta_{uvw}$ on $l_1$ is studied by Panigrahi and Srivastava \cite{pds1} and analogusly, the upper triangular case is studied by Altundag and Abay \cite{filomat1}. The spectrum and fine spectrum of the of the generalised difference operator $\vartriangle^r_v$ over the sequence spaces $c_0$ and $l_1$ have been studied by Dutta and Baliarsingh (see \cite{rv_c0, rv_l1}). In 2017, Birbonshi and Srivastava \cite{bir} have studied various spectral properties of $n$ band triangular matrices of constant bands. Recently El-Shabrawy and Abu-Janah \cite{bv0h} studied the spectrum of the operators $B(r, s)$ and $\Delta_{ab}$ over the sequence space $bv_0$ and $h.$

It has been observed that, the entries in each band of the band matrices chosen by the earlier researchers are either constant sequences or convergent sequences. In the present paper, we have considered an operator which is represented by a triple band lower triangular matrix defined over the sequence space $l_p (1 \leq p < \infty),$ where the bands are taken as periodic sequences of period two. The operator can be represented by the matrix

\begin{equation} \label{mainop3}
   B(r_1,r_2;s_1,s_2;t_1,t_2)=
  \left( {\begin{array}{cccccc}
   r_1   & 0      & 0      & 0      & 0 & \cdots \\
   s_1   & r_2    & 0      & 0      & 0 & \cdots \\
   t_1   & s_2    & r_1    & 0      & 0 & \cdots \\
   0     & t_2    & s_1    & r_2    & 0 & \cdots \\
   0     & 0      & t_1    & s_2    & r_1 & \cdots \\
  \vdots & \vdots & \vdots & \vdots & \vdots & \ddots 
  \end{array} } \right)
\end{equation}
where $r_1, r_2$ are complex numbers, $s_1, s_2$ are non-zero complex numbers and $t_1, t_2$ are complex numbers such that, either $t_1, t_2$ both are non-zero or both are equal to zero. For the sake of convenience we will use the notation $B$ instead of $B(r_1,r_2;s_1,s_2;t_1,t_2)$. Several operators whose spectrum and fine spectrum studied previously can be derived from the above operator with perticular choices of $r_1, r_2,$ $s_1, s_2$ and $t_1, t_2$. The purpose of this work is to obtain the spectrum and fine spectrum of the operator $B$ on the sequence space $l_p(1<p<\infty).$ The theory of system of linear difference equations plays an important role in our work. For detail study of difference equation we refer the book \cite{ela}.

The rest of the paper is organised as follows. Section 2 deals with some notations and review of few concepts of operator theory. Section 3 and 4 contains our results on the spectrum and some spectral classification of the operator $B$ over the sequence space $l_p(1<p<\infty)$ and $l_1$ respectively. Some examples are provided in Section 5 in support of our results, and finally, in Section 6 we conclude our paper by mentioning some previous results which can be derived from our results.

\section{Preliminaries and Notations}
In this paper, the sequence space $l_p (1 \leq p < \infty)$ represents the set of all $p$-absolutely summable sequences of real or complex numbers. All the infinite sequences and matrices are indexed by the set of natural numbers $\mathbb{N}.$ Let $T: X \rightarrow Y$ be a bounded linear operator where $X$ and $Y$ are complex Banach spaces. Then the range space of $T$ and null space of $T$ are denoted by $R(T)$ and $N(T)$ respectively. The adjoint of $T$, denoted by $T^*$, is a bounded linear operator $T^*: Y^* \rightarrow X^*$ which is defined by
\begin{equation*}
(T^* \phi)(x) = \phi(Tx) \ \ \ \mbox{for all} \ \phi \in Y^* \ \mbox{and} \ x \in X
\end{equation*}
where $X^*$ and $Y^*$ are the dual spaces of $X$ and $Y$ respectively. The set of all bounded linear operator on $X$ onto itself is denoted by $B(X).$ For any operator $T \in B(X),$ the resolvent set of $T$ is the set of all complex numbers $\lambda$ for which the operator $T - \lambda I$ has a bounded inverse in $X$ where $I$ is the identity operator in $X.$ The resolvent set of $T$ is denoted by $\rho(T,X).$ The compliment of the resolvant set in the complex plane $\mathbb{C}$ is called the spectrum of $T$ and it is denoted by $\sigma(T,X)$, i.e.,
\begin{equation*}
\sigma(T,X) = \{\lambda \in \mathbb{C}: (T- \lambda I) \ \mbox{has no inverse in} \  B(X)\}.
\end{equation*}
The spectrum $\sigma(T,X)$ can be classified into three disjoint sets as follows:

\noindent The set of points $\lambda \in \mathbb{C}$ for which $N(T- \lambda I) \neq \{0\}$ is called the point spectrum of $T$ and it is denoted by $\sigma_p(T,X)$. The set of points $\lambda \in \mathbb{C}$ for which $N(T- \lambda I) = \{0\}$ and $\overline{R(T- \lambda I)}= X $ but ${R(T- \lambda I)}\neq X $ is called the continuous spectrum of $T.$ It is denoted by $\sigma_c(T,X).$ The set of points $\lambda \in \mathbb{C}$ for which $N(T- \lambda I) = \{0\}$ and $\overline{R(T- \lambda I)}\neq X $ is called the residual spectrum of $T$ and it is denoted by $\sigma_r(T,X)$. 
The three sets $\sigma_p(T,X),\sigma_c(T,X),\sigma_r(T,X)$ are disjoint and their union is the whole spectrum $\sigma(T,X).$

Next we mention some more subdivisions of the spectrum. For any $T \in B(X)$ the approximate point spectrum, defect spectrum, and compression spectrum of $T$ is denoted by $\sigma_{ap}(T, X),$ $\sigma_{\delta}(T, X)$ and $\sigma_{co}(T, X)$ respectively and defined as,
\begin{eqnarray*}
\sigma_{ap}(T, X) &=& \{\lambda \in \mathbb{C} : (T - \lambda I) \ \ \mbox{is not bounded below} \},\\
\sigma_{\delta}(T, X) &=& \{\lambda \in \mathbb{C} : (T - \lambda I) \ \ \mbox{is not surjective} \},\\
\sigma_{co}(T, X) &=& \{\lambda \in \mathbb{C} : \overline{R(T - \lambda I)} \neq X \}.
\end{eqnarray*}
The two sets $\sigma_{ap}(T, X)$ and $\sigma_{\delta}(T, X)$ form a subdivision (not necessarily disjoint) of spectrum, i.e.,
\begin{equation*}
\sigma(T,X) = \sigma_{ap}(T, X) \cup \sigma_{\delta}(T, X)
\end{equation*}
and $\sigma_{ap}(T, X),$ $\sigma_{co}(T, X)$ also form a subdivision (not necessarily disjoint)  of spectrum, i.e.,
\begin{equation*}
\sigma(T,X) = \sigma_{ap}(T, X) \cup \sigma_{co}(T, X).
\end{equation*}
The following results are useful in this context.

\begin{proposition} \cite[p. 28]{nonlinear} \label{proposition}
The following relations hold for an operator $T \in B(X)$
\begin{enumerate}
\item[(a)] $\sigma (T^{*}, X^*) = \sigma (T,X),$ 
\item[(b)] $\sigma_{ap} (T^{*}, X^*) = \sigma_{\delta} (T, X),$ 
\item[(c)] $\sigma_{p} (T^{*}, X^*) = \sigma_{co} (T, X),$ 
\item[(d)] $\sigma (T, X) = \sigma_{ap} (T, X) \cup \sigma_{p} (T^{*}, X^*) = \sigma_{p} (T, X) \cup \sigma_{ap} (T^{*}, X^*).$
\end{enumerate}

\end{proposition}
Depending upon the set $R(T_{\lambda})$ and the inverse $T_{\lambda}^{-1}$ of an operator $T \in B(X)$ where $T_\lambda = T- \lambda I,$ Goldberg \cite{goldberg} classified the spectrum and the resolvant set of $T$ which is given in the following table.

\begin{table}[h]  \label{table1}
\begin{center}
\begin{scriptsize}
\begin{tabular}{ |c c| c| c| c| } 
 \hline
  &  & 1 & 2 & 3 \\

  & &   $T_{\lambda}^{-1}$ exists and bounded &  $T_{\lambda}^{-1}$ exists and unbounded &  $T_{\lambda}^{-1}$ does not exist \\ 
 \hline
 A &   $R(T_{\lambda}) = X$ &  $\lambda \in \rho(T, X)$ & --- & $ \begin{aligned} \lambda &\in \sigma_p(T, X)\\ \lambda &\in \sigma_{ap}(T, X) \end{aligned}$ \\
  \hline  
B &  $\overline{R(T_{\lambda})} = X$ &   $\lambda \in \rho(T, X)$& $ \begin{aligned} \lambda &\in \sigma_c(T, X)\\ \lambda &\in \sigma_{\delta}(T, X) \\ \lambda &\in \sigma_{ap}(T, X) \end{aligned}$ & $ \begin{aligned} \lambda &\in \sigma_p(T, X)\\ \lambda &\in \sigma_{\delta}(T, X) \\ \lambda &\in \sigma_{ap}(T, X) \end{aligned}$ \\
 \hline
 C &  $\overline{R(T_{\lambda})} \neq X$ &  $ \begin{aligned} \lambda &\in \sigma_r(T, X)\\ \lambda &\in \sigma_{\delta}(T, X) \\ \lambda &\in \sigma_{co}(T, X) \end{aligned}$& $ \begin{aligned} \lambda &\in \sigma_r(T, X)\\ \lambda &\in \sigma_{\delta}(T, X) \\ \lambda &\in \sigma_{co}(T, X)\\ \lambda &\in \sigma_{ap}(T, X) \end{aligned}$&$ \begin{aligned} \lambda &\in \sigma_p(T, X)\\ \lambda &\in \sigma_{\delta}(T, X) \\ \lambda &\in \sigma_{co}(T, X)\\ \lambda &\in \sigma_{ap}(T, X)\\  \end{aligned}$ \\
 \hline
\end{tabular}
\end{scriptsize}
\caption{Subdivisions of spectrum of a linear operator}
\end{center}
\end{table}

If we combine the possibilities A, B, C and 1, 2, 3 then nine different states are created. These are labelled by  $A_1$, $A_2$, $A_3$, $B_1$, $B_2$, $B_3,$ $C_1$, $C_2$, $C_3.$ For example, for any $\lambda \in \mathbb{C},$ $T-\lambda I \in A_1$ or $T-\lambda I \in B_1$ then $\lambda \in \rho(T, X).$ If the operator $T$ is in state $C_1$ for example, then $(T- \lambda I)^{-1}$ exists and bounded and $\overline{R(T_{\lambda})} \neq X$ and we write $\lambda \in C_1 \sigma(T, X).$

Let $A=(a_{nk})$ be an infinite matrix of complex numbers and $\lambda$ and $\mu$ be two sequence spaces. Then, the matrix $A$ defines a matrix mapping from $\lambda$ into $\mu$ if for every sequence $x=(x_k) \in \lambda$ the sequence $Ax = \{(Ax)_n\}$ is in $\mu$ where
\[(Ax)_n = \sum\limits_k a_{nk}x_k, \ n \in \mathbb{N} \]
and it is denoted by $A: \lambda \rightarrow \mu.$ We denote the class of all matrices $A$ such that $A: \lambda \rightarrow \mu$ by $(\lambda, \mu).$

\begin{lemma} \cite[p. 59]{goldberg} \label{denserange}
The bounded linear operator $T : X \rightarrow Y$ has dense range if and only if $A^*$ is one to one
\end{lemma}

\begin{lemma} \cite[p. 126]{wilansky} \label{boundedl1}
A matrix $A=(a_{nk})$ gives rise to a bounded linear operator $T \in B(l_1)$ from $l_1$ to itself if and only if the supremum of $l_1$ norms of the columns of A is bounded.
\end{lemma}

\begin{lemma} \cite[p. 126]{wilansky} \label{boundedl}
A matrix $A=(a_{nk})$ gives rise to a bounded linear operator $T \in B(l_\infty)$ from $l_\infty$ to itself if and only if the supremum of $l_1$ norms of the rows of A is bounded.
\end{lemma}

\begin{lemma} \cite[p. 174, Theorem 9]{maddox} \label{boundedlp}
Let $1 < p < \infty$ and suppose $A \in (l_1, l_1) \cap (l_\infty, l_\infty)$. Then $A \in  (l_p, l_p).$
\end{lemma}

\begin{remark}
Throughout our work, if $z$ is a complex number then $\sqrt{z}$ means the square root of $z$ with non-negative real part. If $Re(\sqrt{z}) = 0$ then $\sqrt{z}$ means the square root of $z$ with $Im(z) \geq 0.$
\end{remark}

\section{Spectra of $B$ on $l_p (1 < p < \infty)$}

\begin{theorem}\label{norm}
The operator $B: l_p \rightarrow l_p$ is a bounded linear operator and $B_1 \leq \|B\|_{l_p} \leq B_2$ where
\begin{eqnarray*}
B_1 &=& \max \{(|r_1|^p + |s_1|^p + |t_1|^p)^{\frac{1}{p}} , (|r_2|^p + |s_2|^p + |t_2|^p)^{\frac{1}{p}} \},\\
B_2 &=& \max \{|r_1|, |r_2|\} + \max \{|s_1|, |s_2|\} + \max \{|t_1|, |t_2|\} .
\end{eqnarray*}
\end{theorem}

\begin{proof}
The linearity of $B$ is easy to check so we drop it. For any $(x_k) \in l_p$ with $x_{-1} = x_0 = 0,$ we have
\begin{eqnarray*}
\|Bx\|_{l_p} &=& [|t_1x_{-1}+s_2x_0+r_1x_1|^p+|t_2x_0+s_1x_1+r_2x_2|^p+|t_1x_1+s_2x_2+r_1x_3|^p+\cdots]^{\frac{1}{p}}\\
            &\leq& (\max \{|r_1|, |r_2|\} + \max \{|s_1|, |s_2|\} + \max \{|t_1|, |t_2|\}) \|x\|_{l_p}.
\end{eqnarray*}

This shows that $\|B\|_{l_p} \leq B_2.$ Also let $e^{(k)}$ be a sequence in $l_p$ whose $k$-th component is one and other components are zero. Then 
\begin{equation*}
\|B\|_{l_p} \geq \frac{\|Be^{(k)}\|_{l_p}}{\|e^{(k)}\|_{l_p}} = \left\{
\begin{aligned}
&(|r_1|^p + |s_1|^p + |t_1|^p)^{\frac{1}{p}}, \ k \ \mbox{is odd}&\\
&(|r_2|^p + |s_2|^p + |t_2|^p)^{\frac{1}{p}}, \ k \ \mbox{is even}.&
\end{aligned}
\right.
\end{equation*}
This proves the other part.
\end{proof}

\begin{theorem} \label{spectrum_l1}
Let $s_1, s_2$ be two complex numbers such that $\sqrt{s_1^2} = s_1$ and $\sqrt{s_2^2} = s_2$ then the spectrum of $B$ on $l_p$ is given by 
\begin{equation*}
\sigma(B, l_p) = \left\lbrace \lambda \in \mathbb{C} : \left| \frac{2 (r_1 - \lambda) (r_2 - \lambda)}{s_1s_2 - t_1(r_2 - \lambda)- t_2(r_1 - \lambda) + \sqrt{\chi}} \right| \leq 1 \right\rbrace
\end{equation*} 
where 
\begin{eqnarray} \label{chi}
\chi &=& t_1^2(r_2 - \lambda)^2 + t_2^2(r_1 - \lambda)^2 + s_1^2 s_2^2 - 2t_1t_2(r_1 - \lambda) (r_2 - \lambda)\\
&& -2s_1s_2t_1(r_2 - \lambda) -2s_1s_2t_2(r_1 - \lambda). \nonumber
\end{eqnarray}
\end{theorem}

\begin{proof}
Let $S = \left\lbrace \lambda \in \mathbb{C} : \left| \frac{2 (r_1 - \lambda) (r_2 - \lambda)}{s_1s_2 - t_1(r_2 - \lambda)- t_2(r_1 - \lambda) + \sqrt{\chi}} \right| \leq 1 \right\rbrace$ and let $\lambda \notin S.$ This implies that, $\lambda \notin \{r_1, r_2\}$ and $(B - \lambda I)$ has an inverse and it is of the form
\begin{equation} \label{eq_inverse}
   (B-\lambda I)^{-1}=
  \left( {\begin{array}{ccccc}
   a_1   & 0      & 0      & 0      & \cdots \\
   a_2   & b_1    & 0      & 0      & \cdots \\
   a_3   & b_2    & a_1    & 0      & \cdots \\
   a_4   & b_3    & a_2    & b_1    & \cdots \\
  \vdots & \vdots & \vdots & \vdots & \ddots 
  \end{array} } \right)
\end{equation}
where 
\begin{equation} \label{eq_a1}
\left.\begin{aligned}
t_1 a_{2k-1} + s_2 a_{2k} + (r_1 - \lambda)a_{2k+1} &= &0\\
t_2 a_{2k} + s_1 a_{2k+1} + (r_2 - \lambda)a_{2k+2} &= &0
 \end{aligned}\right\}
\end{equation}
and 
\begin{equation} \label{eq_b1}
\left.\begin{aligned}
t_2 b_{2k-1} + s_1 b_{2k} + (r_2 - \lambda)b_{2k+1} &= &0\\
t_1 b_{2k} + s_2 b_{2k+1} + (r_1 - \lambda)b_{2k+2} &= &0
 \end{aligned}\right\}
\end{equation}
for $k \in \mathbb{N}$ with 
\begin{equation*}
a_1 = \frac{1}{r_1 - \lambda} , \ \ \  a_2 = - \frac{s_1}{(r_1 - \lambda) (r_2 - \lambda)},
\end{equation*}
\begin{equation*}
b_1 = \frac{1}{r_2 - \lambda} , \ \ \  b_2 = - \frac{s_2}{(r_1 - \lambda) (r_2 - \lambda)}.
\end{equation*}
The equations in \eqref{eq_a1} imply
\begin{eqnarray*}
a_{2k+1} &=& -\frac{t_1}{r_1 - \lambda} a_{2k-1} - \frac{s_2}{r_1 - \lambda} a_{2k}\\\\
a_{2k+2} &=& \frac{s_1t_1}{(r_1 - \lambda)(r_2 - \lambda)}a_{2k-1} + \left( -\frac{t_2}{r_2 - \lambda} + \frac{s_1s_2}{(r_1 - \lambda)(r_2 - \lambda)}\right) a_{2k}.
\end{eqnarray*}
The above equations can be written as 
\begin{equation} \label{maineqn}
x_{k+1} = A_1 x_k, \ \ k \in \mathbb{N},
\end{equation}
where 
\begin{equation*} 
   A_1=
  \left( {\begin{array}{cc}
   -\frac{t_1}{r_1 - \lambda}  & - \frac{s_2}{r_1 - \lambda}  \\\\    
   \frac{s_1t_1}{(r_1 - \lambda)(r_2 - \lambda)}   & -\frac{t_2}{r_2 - \lambda} + \frac{s_1s_2}{(r_1 - \lambda)(r_2 - \lambda)}
  \end{array} } \right) , \ \ x_k=
  \left( {\begin{array}{c}
   a_{2k-1}   \\   
   a_{2k}  
  \end{array} } \right) \ \ \mbox{for} \ \ k\in \mathbb{N}.
\end{equation*}

In order to obtain the solution of the system of difference equations \eqref{maineqn}, we need to find the eigenvalues and eigenvectors of $A_1$. We consider two cases here. In case 1 we consider $A_1$ is diagonalisable and in case 2 we consider $A_1$ is not diagonalisable.

\textit{Case 1} : $A_1$ is diagonalisable .

To find the eigenvalues of $A_1,$ let $\det (A_1 - \alpha I) = 0$. Which gives the characteristic equation 
\begin{equation} \label{char1}
f(\alpha) = \alpha^2 + \left(\frac{t_1}{r_1 - \lambda} + \frac{t_2}{r_2 - \lambda} -  \frac{s_1s_2}{(r_1 - \lambda)(r_2 - \lambda)}\right) \alpha + \frac{t_1t_2}{(r_1 - \lambda)(r_2 - \lambda)} = 0.
\end{equation} 
The roots of the equation are
\begin{equation*}
\alpha_1 = \frac{s_1s_2 - t_1(r_2 - \lambda)- t_2(r_1 - \lambda) + \sqrt{\chi}}{2 (r_1 - \lambda)(r_2 - \lambda)}, \ \ \ \alpha_2 = \frac{s_1s_2 - t_1(r_2 - \lambda)- t_2(r_1 - \lambda) - \sqrt{\chi}}{2 (r_1 - \lambda)(r_2 - \lambda)}
\end{equation*}
where $\chi$ is mentioned in \eqref{chi}. Since $A_1$ is diagonalisable and not a scalar multiple of the identity matrix, so $\alpha_1 \neq \alpha_2$. This implies $\chi \neq 0.$ Let $f = 
  \left( {\begin{array}{c}
   f_1  \\   
   f_2
  \end{array} } \right)$ 
be the eigenvector corresponding to $\alpha_1.$ Then $(A_1 - \alpha_1 I)f = 0.$ This gives
\begin{equation*}
\left( {\begin{array}{cc}
   \frac{-s_1s_2 + t_2(\ar)-t_1(\br) - \sqrt{\chi}}{2(\ar)(\br)}  & - \frac{s_2}{\ar}  \\\\    
   \frac{s_1t_1}{(r_1 - \lambda)(r_2 - \lambda)}   & \frac{s_1s_2 - t_2(\ar)+t_1(\br) - \sqrt{\chi}}{2(\ar)(\br)}
  \end{array} } \right) f = \left( {\begin{array}{c}
  0 \\   
  0
  \end{array} } \right).
\end{equation*}
Then
\begin{equation*}
f=(f_1, \ f_2)^t = \left(1, \ \frac{-s_1s_2 + t_2(\ar)-t_1(\br) - \sqrt{\chi}}{2s_2(\br)}  \right)^t
\end{equation*}
is the eigenvector corresponding to $\alpha_1.$ A similar calculation shows that 
\begin{equation*}
g=(g_1, \ g_2)^t = \left(1, \ \frac{-s_1s_2 + t_2(\ar)-t_1(\br) + \sqrt{\chi}}{2s_2(\br)}  \right)^t
\end{equation*}
is the eigenvector corresponding to $\alpha_2.$
Hence the general solution of \eqref{maineqn} is given by (see \cite[p. 137]{ela})
\begin{equation} \label{dia_sol}
x_k = 
  \left( {\begin{array}{c}
   a_{2k-1}   \\   
   a_{2k}
  \end{array} } \right)
= c_1 \left( {\begin{array}{c}
   f_1  \\   
   f_2
  \end{array} } \right) \alpha_1^k + c_2 \left( {\begin{array}{c}
   g_1  \\   
   g_2
  \end{array} } \right) \alpha_2^k, \ \ k \in \mathbb{N},
\end{equation}
where $c_1,$ $c_2$ are constants which can be determined from the relation
\begin{equation*}
x_1 = c_1 \left( {\begin{array}{c}
   f_1  \\   
   f_2
  \end{array} } \right) \alpha_1 + c_2 \left( {\begin{array}{c}
   g_1  \\   
   g_2
  \end{array} } \right) \alpha_2 = 
\left( {\begin{array}{c}
   a_1  \\   
   a_2
  \end{array} } \right) = \left( {\begin{array}{c}
   \frac{1}{r_1 - \lambda}   \\\\   
   - \frac{s_1}{(r_1 - \lambda) (r_2 - \lambda)}  
  \end{array} } \right).
\end{equation*}
By solving the above system of equations we obtain, if $t_1 \neq 0,$ $t_2 \neq 0$ then $\alpha_1 \neq 0,$ $\alpha_2 \neq 0$ and
\begin{equation*}
c_1= \frac{s_1g_1 + (\br)g_2}{\alpha_1(\ar)(\br)(f_1g_2 - g_1f_2)}, \ \ c_2= - \frac{s_1f_1 + (\br)f_2}{\alpha_2(\ar)(\br)(f_1g_2 - g_1f_2)}.
\end{equation*}
Also if $t_1 = t_2 = 0$ then $\alpha_1 \neq 0,$ $\alpha_2 = 0$ and $c_1 = \frac{\br}{s_1s_2}.$ It can be easily verified that, $c_1 \neq 0$ for both the cases $t_1 \neq 0,$ $t_2 \neq 0$ and $t_1 = t_2 = 0.$

\noindent Since $\lambda \notin S$ so $|\alpha_1| <1.$ Now we show that $|\alpha_2|<1.$ If $s_1s_2 - t_1(r_2 - \lambda)- t_2(r_1 - \lambda) = 0$ then, $|\alpha_2| = |\alpha_1|<1.$ Let $s_1s_2 - t_1(r_2 - \lambda)- t_2(r_1 - \lambda) \neq 0.$ Since $|\alpha_1| <1,$ we have
\begin{equation*}
\left|1+ \sqrt{\frac{\chi}{(s_1s_2 - t_1(r_2 - \lambda)- t_2(r_1 - \lambda))^2}}  \right| < \left|\frac{2 (r_1 - \lambda) (r_2 - \lambda)}{s_1s_2 - t_1(r_2 - \lambda)- t_2(r_1 - \lambda)}  \right|.
\end{equation*}
Since $|1- \sqrt{z}| \leq |1 + \sqrt{z}|$ for all $z \in \mathbb{C},$ we have

\begin{equation*}
\left|1- \sqrt{\frac{\chi}{(s_1s_2 - t_1(r_2 - \lambda)- t_2(r_1 - \lambda))^2}}  \right| < \left|\frac{2 (r_1 - \lambda) (r_2 - \lambda)}{s_1s_2 - t_1(r_2 - \lambda)- t_2(r_1 - \lambda)}  \right|.
\end{equation*}
This proves that $|\alpha_2|<1.$ Hence from \eqref{dia_sol} it follows that, $(a_k) \in l_1.$

\textit{Case 2:} $A_1$ is not diagonalisable.

In this case we have
 \[\alpha_1 = \alpha_2 = \alpha =  \frac{s_1s_2 - t_1(r_2 - \lambda)- t_2(r_1 - \lambda)}{2(r_1 - \lambda)(r_2 - \lambda)} \ \ \mbox{and} \ \ \chi=0.\]
 Also if $t_1 = t_2 = 0$ then from \eqref{char1} it follows that one eigenvalue is zero and other is non-zero which is not possible in this case. Therefore $t_1 \neq 0$ and $t_2 \neq 0.$

 \noindent Let $h= \left( {\begin{array}{c}
   h_1   \\   
   h_2 
  \end{array} } \right)$ be the corresponding eigenvector of $\alpha.$ Then $(A_1 - \alpha I)h=0$ gives
  \begin{equation*}
\left( {\begin{array}{cc}
   \frac{-s_1s_2 + t_2(\ar)-t_1(\br)}{2(\ar)(\br)}  & - \frac{s_2}{\ar}  \\\\    
   \frac{s_1t_1}{(r_1 - \lambda)(r_2 - \lambda)}   & \frac{s_1s_2 - t_2(\ar)+t_1(\br)}{2(\ar)(\br)}
  \end{array} } \right) h = \left( {\begin{array}{c}
  0 \\   
  0
  \end{array} } \right).
\end{equation*}
Then $$h=(h_1 , \ h_2)^t = \left(1, \ \frac{-s_1s_2 + t_2(\ar)-t_1(\br)}{2s_2(\br)}  \right)^t$$ is the eigenvector corresponding to $\alpha.$ Let $u = (u_1, \ u_2)^t$ be the generalised eigenvector of $\alpha.$ Then
\begin{equation*}
(A_1 - \alpha I)u = h.
\end{equation*}
The Jordan form of $A_1$ is given by
  \begin{equation*}
  P^{-1}A_1P = J
  \end{equation*}
where
\begin{equation*} 
   J=
  \left( {\begin{array}{cc}
   \alpha  & 1  \\   
   0  & \alpha
  \end{array} } \right) \ \ \mbox{and} \ \
  P=
  \left( {\begin{array}{cc}
   h_1  & u_1  \\   
   h_2  & u_2
  \end{array} } \right).
\end{equation*}
Hence the general solution of \eqref{eq_a1} can be written as (see \cite[p. 145]{ela})
\begin{equation*}
x_k  = PJ^k \tilde{c}, \ k \in \mathbb{N}.
\end{equation*}
Where $\tilde{c} = \left( {\begin{array}{c}
   \tilde{c_1}   \\   
   \tilde{c_2} 
  \end{array} } \right) \ $  \ and \ $J^k = \left( {\begin{array}{cc}
   \alpha^k  & k \alpha^{k-1}  \\   
   0  & \alpha^k
  \end{array} } \right).$ Therefore the solution $x_k$ is of the form
  \begin{equation} \label{nondia_sol}
  x_k = 
  \left( {\begin{array}{c}
   a_{2k-1}   \\   
   a_{2k}
  \end{array} } \right)
  =
  \left( {\begin{array}{c}
   \tilde{c_1} h_1 \alpha^k + \tilde{c_2}(h_1 k \alpha^{k-1} + u_1 \alpha^k)   \\   
   \tilde{c_1} h_2 \alpha^k + \tilde{c_2}(h_2 k \alpha^{k-1} + u_2 \alpha^k)
  \end{array} } \right)
  \end{equation}
with 
\begin{equation*}
x_1  = \left( {\begin{array}{c}
   a_{1}   \\   
   a_{2}
  \end{array} } \right)
=  \left( {\begin{array}{c}
   \tilde{c_1} h_1 \alpha + \tilde{c_2}(h_1  + u_1 \alpha)   \\   
   \tilde{c_1} h_2 \alpha + \tilde{c_2}(h_2  + u_2 \alpha)
  \end{array} } \right) = 
\left( {\begin{array}{c}
   \frac{1}{r_1 - \lambda}   \\\\  
   - \frac{s_1}{(r_1 - \lambda) (r_2 - \lambda)}  
  \end{array} } \right).
\end{equation*}
Since $t_1 \neq 0,$ $t_2 \neq 0$ then $\alpha \neq 0$ and
\begin{equation*}
\tilde{c_1}= \frac{s_1(h_1 + u_1 \alpha) + (\br)(h_2 + u_2 \alpha)}{\alpha^2(\ar)(\br)(h_1u_2 - u_1h_2)}, \ \ \tilde{c_2}= - \frac{s_1h_1 + (\br)h_2}{\alpha(\ar)(\br)(h_1u_2 - u_1h_2)}.
\end{equation*}
Since $|\alpha| = \left| \frac{s_1s_2 - t_1(r_2 - \lambda)- t_2(r_1 - \lambda)}{2(r_1 - \lambda)(r_2 - \lambda)} \right| < 1,$ from \eqref{nondia_sol}, we have $(a_k) \in l_1.$

Now from the equations in \eqref{eq_b1} we have for $k\in \mathbb{N} $
\begin{eqnarray*}
b_{2k+1} &=& -\frac{t_2}{r_2 - \lambda} b_{2k-1} - \frac{s_1}{r_2 - \lambda} b_{2k}\\
b_{2k+2} &=& \frac{s_2t_2}{(r_1 - \lambda)(r_2 - \lambda)}b_{2k-1} + \left( -\frac{t_1}{r_1 - \lambda} + \frac{s_1s_2}{(r_1 - \lambda)(r_2 - \lambda)}\right) b_{2k}.
\end{eqnarray*}
The above equations can be written as 
\begin{equation}
y_{k+1} = A_2 y_k, \ \ k \in \mathbb{N},
\end{equation}
where 
\begin{equation*} 
   A_2=
  \left( {\begin{array}{cc}
   -\frac{t_2}{r_2 - \lambda}  & - \frac{s_1}{r_2 - \lambda}  \vspace{0.4cm}  \\     
   \frac{s_2t_2}{(r_1 - \lambda)(r_2 - \lambda)}  & -\frac{t_1}{r_1 - \lambda} + \frac{s_1s_2}{(r_1 - \lambda)(r_2 - \lambda)}
  \end{array} } \right) , \ \ \ 
  y_k= \left( {\begin{array}{c}
   b_{2k-1}\\   
   b_{2k}  
  \end{array} } \right), \ \ k \in \mathbb{N}.
\end{equation*}
The characteristic equation of $A_2$ is $\det(A_2 - \alpha I) = 0.$ Which gives
\begin{equation*} 
\alpha^2 + \left(\frac{t_1}{r_1 - \lambda} + \frac{t_2}{r_2 - \lambda} -  \frac{s_1t_1}{(r_1 - \lambda)(r_2 - \lambda)}\right) \alpha + \frac{t_1t_2}{(r_1 - \lambda)(r_2 - \lambda)} = 0.
\end{equation*}
Since the characteristic equation of the matrices $A_1$ and $A_2$ are same so the eigenvalues of $A_2$ are $\alpha_1$ and $\alpha_2$. Then by similar method as above it can be shown that $(b_k) \in l_1.$ This implies $(B - \lambda I)^{-1} \in (l_1, l_1).$ Also since $(a_k) \in l_1$ and $(b_k) \in l_1$, the supremum of the $l_1$ norms of the rows of $(B - \lambda I)^{-1}$ is finite. This proves that $(B - \lambda I)^{-1} \in (l_\infty , l_\infty).$ Therefore $(B - \lambda I)^{-1} \in (l_p, l_p)$. Hence $\lambda \notin \sigma(B, l_p)$ and this shows that $\sigma(B, l_p) \subseteq S.$

Conversely, let $\lambda \in S.$ If $\lambda = r_1$ or $\lambda = r_2$ then $B - \lambda I$ does not have dense range, so it is not invertible. Hence let $\lambda \notin \{r_1, r_2\}.$ Then
\begin{equation*}
   (B-\lambda I)^{-1}=
  \left( {\begin{array}{ccccc}
   a_1   & 0      & 0      & 0      & \cdots \\
   a_2   & b_1    & 0      & 0      & \cdots \\
   a_3   & b_2    & a_1    & 0      & \cdots \\
   a_4   & b_3    & a_2    & b_1    & \cdots \\
  \vdots & \vdots & \vdots & \vdots & \ddots 
  \end{array} } \right)
\end{equation*}
where $(a_k)$ and $(b_k)$ are given in \eqref{eq_a1} and \eqref{eq_b1} respectively. If $\chi = 0$ then $\left|\frac{s_1s_2 - t_1(r_2 - \lambda)- t_2(r_1 - \lambda)}{2(r_1 - \lambda I)(r_2 - \lambda I)}  \right| \geq 1.$ Therefore from \eqref{nondia_sol} it follows that $(a_k) \notin l_p.$ Also if $\chi \neq 0$ then from \eqref{dia_sol} we have $|\alpha_1| \geq 1$ and $c_1 \neq 0.$ This also implies that $(a_k) \notin l_p.$ Now $x = (1, 0,0 ,\cdots) \in l_p$ but $(B - \lambda I)^{-1}x = (a_1, a_2, \cdots) \notin l_p.$ This shows that $(B - \lambda I)^{-1} \notin B(l_p).$ Hence $S \subseteq \sigma(B, l_p).$ 
\end{proof}

\begin{theorem} \label{pspectrum_l1}
The point spectrum of $B$ on $l_p$ is given by
\begin{equation*}
\sigma_p(B, l_p) = \emptyset.
\end{equation*}
\end{theorem}

\begin{proof}
Suppose $(B- \lambda I) x = 0$ for $x \neq \theta$ in $l_p.$ This gives the following equations
\begin{eqnarray*}
r_1 x_1 &=& \lambda x_1 \\
s_1 x_1 + r_2 x_2 &=& \lambda x_2 \\
t_1 x_1 + s_2 x_2 + r_1 x_3 &=& \lambda x_3 \\
t_2 x_2 + s_1 x_3 + r_2 x_4 &=& \lambda x_4\\
    & \vdots &
\end{eqnarray*}
From the above equation it follows that, if $\lambda \notin \{r_1, r_2\}$ then $x_k = 0$ for all $k \in \mathbb{N}$. This implies $\sigma_p(B, l_p) \subseteq \{r_1, r_2\}.$  Now we consider two cases here.\\
\textit{Case 1}: $r_1 = r_2 = r.$

Consider $(B- r I) x = 0$ for $x \neq \theta$. Then we have the following equations
\begin{eqnarray*}
r x_1 &=& r x_1 \\
s_1 x_1 + r x_2 &=& r x_2 \\
t_1 x_1 + s_2 x_2 + r x_3 &=& r x_3 \\
t_2 x_2 + s_1 x_3 + r x_4 &=& r x_4\\
    & \vdots &
\end{eqnarray*}
Since $s_1, s_2$ are non-zero complex numbers, from above equations it follows that $x_k = 0$ for $k \in \mathbb{N}.$ Hence $\sigma_p(B, l_p) = \emptyset$ in this case.

\textit{Case 2}: $r_1 \neq r_2.$

Consider $(B- r_1 I) x = 0$ for $x \neq \theta$. Then we have the following equations
\begin{eqnarray} \label{eigen_r1}
\left.\begin{aligned}
r_1 x_1 &=& r_1 x_1 \\
s_1 x_1 + r_2 x_2 &=& r_1 x_2 \\
t_1 x_1 + s_2 x_2 + r_1 x_3 &=& r_1 x_3 \\
t_2 x_2 + s_1 x_3 + r_2 x_4 &=& r_1 x_4\\
t_1 x_3 + s_2 x_4 + r_1 x_5 &=& r_1 x_5 \\
t_2 x_4 + s_1 x_5 + r_2 x_6 &=& r_1 x_6\\
    & \vdots &
 \end{aligned}\right\}
\end{eqnarray}
Here we consider two subcases.

\textit{Subcase 1}: $t_1 \neq 0,$ $t_2 \neq 0.$

Let $s_1s_2-t_1(r_2 - r_1)\neq 0.$ Then second and third equation of \eqref{eigen_r1} give
\begin{eqnarray} \label{x1x2}
\left.\begin{aligned}
s_1 x_1 + (r_2 - r_1) x_2 &=& 0 \\
t_1 x_1 + s_2 x_2  &=& 0
\end{aligned}\right\}.
\end{eqnarray}
This implies $x_1 = 0$ and $x_2 = 0.$ From fourth and fifth equation of \eqref{eigen_r1} we have
\begin{eqnarray} \label{x3x4}
\left.\begin{aligned}
s_1 x_3 + (r_2 - r_1) x_4 &=& 0 \\
t_1 x_3 + s_2 x_4  &=& 0
\end{aligned}\right\}.
\end{eqnarray}
Then $x_3 = 0$ and $x_4 = 0.$  Proceeding in this way we have $x_k = 0$ for all $k \in \mathbb{N}.$ Which is a contradiction.

Now let  $s_1s_2 - t_1 (r_2 - r_1) = 0$ where $s_1,$ $s_2,$ $t_1$  are non-zero and $r_1 \neq r_2.$  
Then from \eqref{x1x2} we have either $(x_1, x_2) = (0, 0)$ or $x_1 \neq 0,$ $x_2 \neq 0.$ If $(x_1, x_2) = (0, 0)$ then from from \eqref{x3x4} we have either $(x_3, x_4) = (0, 0)$ or $x_3 \neq 0,$ $x_4 \neq 0.$ In this way if we consider every pair $(x_1, x_2),(x_3, x_4), \cdots$ equals to $(0, 0)$ then $x_k = 0$ for all $k \in \mathbb{N}.$ Hence let $(x_p, x_{p+1})$ be the first pair such that $x_p \neq 0$ and $x_{p+1} \neq 0$ for some odd $p \in \mathbb{N}$. This implies
\begin{equation*}
(x_1, x_2) = (x_3, x_4) = \cdots = (x_{p-2}, x_{p-1}) = (0, 0).
\end{equation*}
Then the equations in \eqref{eigen_r1} reduce to 
\begin{eqnarray} \label{xp}
\left.\begin{aligned}
s_1 x_p + (r_2 - r_1)x_{p+1} &=& 0 \\
t_1 x_p + s_2 x_{p+1} &=& 0 \\
t_2 x_{p+1} + s_1 x_{p+2} + (r_2 - r_1) x_{p+3} &=& 0\\
t_1 x_{p+2} + s_2 x_{p+3}  &=& 0 \\
    & \vdots &
 \end{aligned}\right\}
\end{eqnarray}
From 4th equation of \eqref{xp} we have $x_{p+3}= - \frac{t_1}{s_2}x_{p+2}$. Substituting it to 3rd equation of \eqref{xp} we get $t_2s_2 x_{p+1} + (s_1s_2 - t_1(r_2 - r_1))x_{p+2} = 0$ and this implies $t_2s_2x_{p+1} = 0$ which is not possible. Therefore $r_1 \notin \sigma_p(B, l_p).$

\textit{Sub-case 2}: $(t_1, t_2) = (0, 0).$

Here the system of equations \eqref{eigen_r1} reduces to
\begin{eqnarray*}
s_1 x_{2k-1} + (r_2 - r_1) x_{2k} &=& 0\\
s_2 x_{2k} &=& 0,
\end{eqnarray*}
where $k \in \mathbb{N}.$ This implies $x_{2k} = 0$ and consequently $x_{2k-1} = 0$ for all $k \in \mathbb{N}.$ Therefore $r_1 \notin \sigma_p(B, l_p).$

\noindent Similarly as above it can be proved that $r_2 \notin \sigma_p(B ,l_p)$ and this proves the theorem.
\end{proof}

\begin{theorem} \label{pointlp*}
The point spectrum of $B^*$ over $l_p^*$ is given by
\begin{equation*}
\sigma_p(B^*, l_p^*)  = \left\lbrace \lambda \in \mathbb{C} : \left| \frac{2 (r_1 - \lambda) (r_2 - \lambda)}{s_1s_2 - t_1(r_2 - \lambda)- t_2(r_1 - \lambda) + \sqrt{\chi}} \right| < 1 \right\rbrace,
\end{equation*}
where $\chi$ is mentioned in \eqref{chi}.
\end{theorem}

\begin{proof} 
 Let $S_1 = \left\lbrace \lambda \in \mathbb{C} : \left| \frac{2 (r_1 - \lambda) (r_2 - \lambda)}{s_1s_2 - t_1(r_2 - \lambda)- t_2(r_1 - \lambda) + \sqrt{\chi}} \right| < 1 \right\rbrace$ and $\lambda \in S_1$. Consider $(B^* - \lambda I)x = 0$ where $x \neq \theta.$ This implies the following equations
\begin{eqnarray} \label{eigen_*}
\left.\begin{aligned}
r_1 x_1 + s_1 x_2 + t_1 x_3 &=& \lambda x_1 \\
r_2 x_2 + s_2 x_3 + t_2 x_4 &=& \lambda x_2\\
r_1 x_3 + s_1 x_4 + t_1 x_5 &=& \lambda x_3 \\
& \vdots &
\end{aligned}\right\}
\end{eqnarray}
If $\lambda = r_1$ then $(1, 0, 0, \cdots)$ is an eigenvector corresponding to the eigenvalue $r_1$. This implies $r_1 \in \sigma_p(B^*, l_p^*).$ If $\lambda = r_2$ then $(1, - \frac{r_1 - r_2}{s_1}, 0, 0, \cdots)$ is an eigenvector corresponding to the eigenvalue $r_2$. This implies $r_2 \in \sigma_p(B^*, l_p^*).$ Assume $\lambda \notin \{r_1, r_2\}.$ Then \eqref{eigen_*} can be written as 
\begin{eqnarray} \label{eigen_*1}
\left.\begin{aligned}
(r_1 - \lambda) x_{2k -1} + s_1 x_{2k} + t_1 x_{2k+1} &=& 0\\
(r_2 - \lambda) x_{2k } + s_2 x_{2k+1} + t_2 x_{2k+2} &=& 0
\end{aligned}\right\}
\end{eqnarray}
for $k = 1,2,3, \cdots.$ We prove the theorem in two cases.

\textit{Case 1}: $t_1 \neq 0,$ $t_2 \neq 0.$

Equations in \eqref{eigen_*1} implies
\begin{eqnarray*}
x_{2k+1} &=& -\frac{(r_1 - \lambda)}{t_1} x_{2k -1} - \frac{s_1}{t_1} x_{2k}\\
x_{2k+2} &=& \frac{s_2 (r_1 - \lambda)}{t_1 t_2} x_{2k -1} + \left(- \frac{(r_2 - \lambda) }{t_2} + \frac{s_1 s_2}{t_1 t_2}  \right) x_{2k}.
\end{eqnarray*}
Therefore we have the following system of difference equation
\begin{equation*}
X_{k+1} = C X_k, \ \ k=1,2, \cdots,
\end{equation*}
where 
\begin{equation*} 
   C=
  \left( {\begin{array}{cc}
   -\frac{(r_1 - \lambda)}{t_1}  & - \frac{s_1}{t_1}  \\\\    
   \frac{s_2 (r_1 - \lambda)}{t_1 t_2}   & - \frac{(r_2 - \lambda) }{t_2} + \frac{s_1 s_2}{t_1 t_2} 
  \end{array} } \right) , \ \ X_k=
  \left( {\begin{array}{c}
   x_{2k-1}   \\   
   x_{2k}  
  \end{array} } \right) \ \ \mbox{for} \ \ k=1,2, \cdots.
\end{equation*}
The characteristic equation of the matrix $C$ is
\begin{equation*} 
g(\alpha)= \alpha^2 + \left(\frac{r_1 - \lambda}{t_1} + \frac{r_2 - \lambda}{t_2} -  \frac{s_1s_2}{t_1 t_2}\right) \alpha + \frac{(r_1 - \lambda)(r_2 - \lambda)}{t_1t_2} = 0.
\end{equation*}
It can be easily verified that $g(\alpha) = 0$ is the reciprocal equation of $f(\alpha)=0$ in \eqref{char1} and since $t_1 \neq 0,$ $t_2 \neq 0$ and $\lambda \notin \{r_1, r_2\}$ then the roots of both the equations $f(\alpha) = 0$ and $g(\alpha) = 0$ are non zero. Then 
\begin{equation*}
\alpha = \frac{2 (r_1 - \lambda) (r_2 - \lambda)}{s_1s_2 - t_1(r_2 - \lambda)- t_2(r_1 - \lambda) + \sqrt{\chi}}
\end{equation*} 
is an eigenvalue of $C$ and let $f = (f_1, f_2)^t$ be the corresponding eigenvector. Then
\begin{equation*}
  X_k = 
  \left( {\begin{array}{c}
   x_{2k-1}   \\   
   x_{2k}
  \end{array} } \right)
  =
  \left( {\begin{array}{c}
   f_1  \\   
   f_2
  \end{array} } \right) \alpha^k , \ \ k= 1,2, \cdots
  \end{equation*}
is a solution of \eqref{eigen_*1}. Since $\lambda \in S_1$ so $|\alpha| < 1$ and this implies $(x_k) \in l_p^* \cong l_{q}$  where $\frac{1}{p} + \frac{1}{q} = 1.$
This proves that $\lambda \in \sigma_p(B^*, l_p^*).$ Hence 
\begin{equation} \label{l1}
S_1 \subseteq  \sigma_p(B^*, l_p^*).
\end{equation}
 Also $S_1 \subseteq \sigma_p(B^*, l_p^*) \subseteq \sigma(B^*, l_p^*) = \sigma(B, l_p) = S_1 \cup S_2$ where
\begin{equation*}
S_2 =  \left\lbrace \lambda \in \mathbb{C} : \left| \frac{2 (r_1 - \lambda) (r_2 - \lambda)}{s_1s_2 - t_1(r_2 - \lambda)- t_2(r_1 - \lambda) + \sqrt{\chi}} \right| = 1 \right\rbrace.
\end{equation*}
Let $\lambda \in S_2.$ Then the eigenvalues of $C$ are 
\begin{equation*}
\beta_1 = \frac{s_1s_2 - t_1(r_2 - \lambda)- t_2(r_1 - \lambda) - \sqrt{\chi}}{2 t_1t_2}, \ \ \
\beta_2 = \frac{s_1s_2 - t_1(r_2 - \lambda)- t_2(r_1 - \lambda) + \sqrt{\chi}}{2 t_1t_2}
\end{equation*}
where $\beta_1 = \frac{1}{\alpha_1},$ $\beta_2 = \frac{1}{\alpha_2}.$ Let $C$ is diagonalisable. Then the general solution of \eqref{eigen_*1} is of the form
\begin{equation} \label{dia_sol1} 
  \left( {\begin{array}{c}
   x_{2k-1}   \\   
   x_{2k}
  \end{array} } \right)
= c_1 \left( {\begin{array}{c}
   \tilde{f}_1  \\   
   \tilde{f}_2
  \end{array} } \right) \beta_1^k + c_2 \left( {\begin{array}{c}
   \tilde{g}_1  \\   
   \tilde{g}_2
  \end{array} } \right) \beta_2^k, \ \ k \in \mathbb{N},
\end{equation}
where $(\tilde{f}_1, \ \tilde{f}_2)^t$ and $(\tilde{g}_1, \ \tilde{g}_2)^t$ are eigenvectors corresponding to $\beta_1$ and $\beta_2$ respectively and $c_1$ and $c_2$ are arbitrary constants. Since $\lambda \in S_2$ so $\left|\frac{1}{\alpha_1} \right| = 1$ i.e., $|\beta_1| = 1.$ Also from the relation $|1 - \sqrt{z}| \leq |1 + \sqrt{z}|$ for all $z \in \mathbb{C},$ we have $|\alpha_1| = 1$ implies $|\alpha_2| \leq 1.$ Therefore $\left|\frac{1}{\alpha_2} \right| = |\beta_2| \geq 1.$ Using these relations in \eqref{dia_sol1} we get $x_k \nrightarrow 0$ as $k \rightarrow \infty.$ Hence $(x_k) \notin l_p^*.$ Also let $C$ is not diagonalisable. Then $\chi = 0$ and the eigenvalue of $C$ is 
\begin{equation*}
\beta = \frac{s_1s_2 - t_1(r_2 - \lambda)- t_2(r_1 - \lambda)}{2 t_1t_2}.
\end{equation*}
Since $\chi = 0,$
\begin{equation*}
\beta = \frac{(s_1s_2 - t_1(r_2 - \lambda)- t_2(r_1 - \lambda))^2}{2 t_1t_2(s_1s_2 - t_1(r_2 - \lambda)- t_2(r_1 - \lambda))} = \frac{2(r_2 - \lambda)(r_1 - \lambda)}{s_1s_2 - t_1(r_2 - \lambda)- t_2(r_1 - \lambda)}.
\end{equation*}
This implies $|\beta| = 1$ since $\lambda \in S_2.$ Then similar as Theorem \ref{spectrum_l1}, the general solution of \eqref{eigen_*1} is of the form
\begin{equation} \label{nondia_sol1}
  \left( {\begin{array}{c}
   x_{2k-1}   \\   
   x_{2k}
  \end{array} } \right)
  =
  \left( {\begin{array}{c}
   \tilde{c_1} \tilde{h}_1 \beta^k + \tilde{c_2}(\tilde{h}_1 k \beta^{k-1} + \tilde{u}_1 \beta^k)   \\   
   \tilde{c_1} \tilde{h}_2 \beta^k + \tilde{c_2}(\tilde{h}_2 k \beta^{k-1} + \tilde{u}_2 \beta^k)
  \end{array} } \right)
  \end{equation}
where $(\tilde{h}_1, \ \tilde{h}_2)^t$ and $(\tilde{u}_1, \ \tilde{u}_2)^t$ are the eigenvector and generalised eigenvector of $\beta$ respectively and $\tilde{c_1}, \tilde{c_2}$ are arbitrary constants. Since $|\beta| = 1,$ from \eqref{nondia_sol1} it follows that, $x_k \nrightarrow 0$ as $k \rightarrow \infty.$ Hence $(x_k) \notin l_p^*.$ Therefore $\lambda \notin \sigma_p(B^*, l_p^*).$ This proves that $\sigma_p(B^*, l_p^*) = S_1$.

Case 2: $t_1 = t_2 = 0.$

In this case we have
\begin{equation*}
S_1 = \left\lbrace \lambda \in \mathbb{C} : \left| \frac{(\ar)(\br)}{s_1s_2}\right| < 1 \right\rbrace
\end{equation*}
Then \eqref{eigen_*1} reduces to
\begin{eqnarray} \label{teqzero}
\left.\begin{aligned}
(r_1 - \lambda) x_{2k -1} + s_1 x_{2k} &=& 0\\
(r_2 - \lambda) x_{2k } + s_2 x_{2k+1} &=& 0
\end{aligned}\right\}
\end{eqnarray}
for $k = 1,2,3, \cdots$. If $x_1 = 0$ then the above equations imply $x_k = 0$ for all $k \in \mathbb{N}.$ Hence let $x_1 \neq 0.$ Then from \eqref{teqzero} we have
\begin{eqnarray*}
x_{2k} &=& - \frac{\ar}{s_1} \left(\frac{(\ar)(\br)}{s_1s_2} \right)^{k-1} x_1,\\\\
x_{2k+1} &=&  \left(\frac{(\ar)(\br)}{s_1s_2} \right)^{k}x_1
\end{eqnarray*}
for $k = 1,2,3, \cdots.$ Therefore $(x_k) \in l_p^* \cong l_{q}$ if and only if $\left| \frac{(\ar)(\br)}{s_1s_2}\right| <1.$ This proves the required result.

\end{proof}

\begin{theorem} \label{res_cont}
The residual spectrum and continuous spectrum of $B$ over $l_p$ are given by
\begin{enumerate}
\item[(i)] $\sigma_r(B, l_p)=\left\lbrace \lambda \in \mathbb{C} : \left| \frac{2 (r_1 - \lambda) (r_2 - \lambda)}{s_1s_2 - t_1(r_2 - \lambda)- t_2(r_1 - \lambda) + \sqrt{\chi}} \right| < 1 \right\rbrace,$\\
\item[(ii)] $\sigma_c(B, l_p)=\left\lbrace \lambda \in \mathbb{C} : \left| \frac{2 (r_1 - \lambda) (r_2 - \lambda)}{s_1s_2 - t_1(r_2 - \lambda)- t_2(r_1 - \lambda) + \sqrt{\chi}} \right| = 1 \right\rbrace.$
\end{enumerate}
Where $\chi$ is mentioned in \eqref{chi}.
\end{theorem}

\begin{proof}
From Lemma \ref{denserange} it can be easily derived that,
\begin{equation*}
\sigma_r(B, l_p) = \sigma_p(B^*, l_p^*) \setminus \sigma_p(B, l_p).
\end{equation*}
Hence the results for residual spectrum follows from Theorem \ref{pspectrum_l1} and Theorem \ref{pointlp*}. Also since the spectrum is a disjoint union of point spectrum, residual spectrum and continuous spectrum, the other result follows immediately.
\end{proof}

\begin{theorem}\label{goldberg}
The operator $B$ satisfies the following relations,
\begin{enumerate}
\item[(a)]  $A_3 \sigma(B, l_p)=B_3 \sigma(B, l_p)= C_3 \sigma(B, l_p) = \emptyset ,$
\item[(b)] $B_2 \sigma(B, l_p) = \left\lbrace \lambda \in \mathbb{C} : \left| \frac{2 (r_1 - \lambda) (r_2 - \lambda)}{s_1s_2 - t_1(r_2 - \lambda)- t_2(r_1 - \lambda) + \sqrt{\chi}} \right| = 1 \right\rbrace,$
\item[(c)] $ C_2 \sigma(B, l_p) \supseteq \left\lbrace \lambda \in \mathbb{C} : \left| \frac{2 (r_1 - \lambda) (r_2 - \lambda)}{s_1s_2 - t_1(r_2 - \lambda)- t_2(r_1 - \lambda) + \sqrt{\chi}} \right| < 1 \right\rbrace \setminus \{r_1, r_2\},$ 
\item[(d)] $C_1 \sigma(B, l_p) \subseteq \{r_1, r_2\}.$
\end{enumerate}
\end{theorem}

\begin{proof}
From Table 1 we have the following relations
\begin{eqnarray*}
\sigma_p(B, l_p) &=& A_3 \sigma(B, l_p) \cup B_3 \sigma(B, l_p) \cup C_3 \sigma(B, l_p),\\
\sigma_r(B, l_p) &=& C_1 \sigma(B, l_p) \cup C_2 \sigma(B, l_p),\\
\sigma_c(B, l_p) &=& B_2 \sigma(B, l_p). 
\end{eqnarray*}
 The result in $(a)$ follows from the above relations and Theorem \ref{pspectrum_l1}. Also since $\sigma_c(B, l_p) = B_2 \sigma(B, l_p),$ the result in $(b)$ follows from Theorem \ref{res_cont}. Again from the proof of Theorem \ref{spectrum_l1} it follows that, for any $\alpha \in \sigma_r(B, l_p)\setminus \{r_1, r_2\}$ the operator $(B - \alpha I)^{-1} \notin B(l_p)$. This implies that $\sigma_r(B, l_p) \setminus \{r_1, r_2\} \subseteq C_2 \sigma(B, l_p)$ and  $C_1 \sigma(B, l_p) \subseteq \{r_1, r_2\}.$
\end{proof}

\begin{theorem} \label{ap_lp}
The operator $B$ satisfies the following relations
\begin{enumerate}
\item[(a)] $ \sigma_{ap}(B, l_p) \supseteq  \left\lbrace \lambda \in \mathbb{C} : \left| \frac{2 (r_1 - \lambda) (r_2 - \lambda)}{s_1s_2 - t_1(r_2 - \lambda)- t_2(r_1 - \lambda) + \sqrt{\chi}} \right| \leq 1 \right\rbrace\setminus \{r_1, r_2\}, $
\item[(b)] $\sigma_{ap}(B^*, l_p^*) = \left\lbrace \lambda \in \mathbb{C} : \left| \frac{2 (r_1 - \lambda) (r_2 - \lambda)}{s_1s_2 - t_1(r_2 - \lambda)- t_2(r_1 - \lambda) + \sqrt{\chi}} \right| \leq 1 \right\rbrace,$
\item[(c)] $\sigma_{\delta}(B, l_p) = \left\lbrace \lambda \in \mathbb{C} : \left| \frac{2 (r_1 - \lambda) (r_2 - \lambda)}{s_1s_2 - t_1(r_2 - \lambda)- t_2(r_1 - \lambda) + \sqrt{\chi}} \right| \leq 1 \right\rbrace,$
\item[(d)] $\sigma_{co}(B, l_p) = \left\lbrace \lambda \in \mathbb{C} : \left| \frac{2 (r_1 - \lambda) (r_2 - \lambda)}{s_1s_2 - t_1(r_2 - \lambda)- t_2(r_1 - \lambda) + \sqrt{\chi}} \right| < 1 \right\rbrace.$
\end{enumerate}
\end{theorem}

\begin{proof}
From Table 1 we have $\sigma_{ap}(B, l_p) = \sigma(B, l_p) \setminus C_1 \sigma(B, l_p)$ and the result in $(a)$ follows from Theorem \ref{goldberg}. The results in $(b),$ $(c)$ and $(d)$ follow from the relations $(d),$ $(b)$ and $(c)$ in Proposition \ref{proposition} respectively.
\end{proof}

\section{Spectra of $B$ on $l_1$}
In this section we give the results on the fine spectrum of the operator $B$ on the sequence space $l_1.$ First we mention the result on the boundedness of $B.$ 

\begin{theorem}
The operator $B: l_1 \rightarrow l_1$ is a bounded linear operator and
\begin{equation*}
\|B\|_{l_1}  = \max \{|r_1|+|s_1|+|t_1|, \ |r_2| +|s_2|+|t_2|\}.
\end{equation*}
\end{theorem}

\begin{proof}
From Lemma \ref{boundedl1} it follows that $B$ is a bounded linear operator from $l_1$ to itself. For any $(x_k)\in l_1$ we have
\[\|Bx\|_{l_1}\leq \max \{|r_1|+|s_1|+|t_1|, \ |r_2| +|s_2|+|t_2|\} \|x\|_{l_1}.\]
This implies $\|B\|_{l_1} \leq \max \{|r_1|+|s_1|+|t_1|, \ |r_2| +|s_2|+|t_2|\}.$ Also let $e^{(k)}$ be a sequence in $l_1$ whose $k$-th component is one and other components are zero. Then similarly as Theorem \ref{norm} it can be derived that
\[\|B\|_{l_1} \geq \max \{|r_1|+|s_1|+|t_1|, \ |r_2| +|s_2|+|t_2|\}.\]
This proves the result.
\end{proof}

Since the spectrum and point spectrum of $B$ on $l_1$ can be derived using the similar arguments used in the case of $l_p,$ we omit the proof and give the statements only.

\begin{theorem} 
Let $s_1, s_2$ be two complex numbers such that $\sqrt{s_1^2} = s_1$ and $\sqrt{s_2^2} = s_2$ then the spectrum of $B$ on $l_1$ is given by 
\begin{equation*}
\sigma(B, l_1) = \left\lbrace \lambda \in \mathbb{C} : \left| \frac{2 (r_1 - \lambda) (r_2 - \lambda)}{s_1s_2 - t_1(r_2 - \lambda)- t_2(r_1 - \lambda) + \sqrt{\chi}} \right| \leq 1 \right\rbrace
\end{equation*} 
where $\chi$ is mentioned in \eqref{chi}.
\end{theorem}

\begin{theorem} 
The point spectrum of $B$ on $l_1$ is given by
\begin{equation*}
\sigma_p(B, l_1) = \emptyset.
\end{equation*}
\end{theorem}

\begin{theorem} 
The point spectrum of $B^*$ over $l_1^*$ is given by
\begin{equation*}
\sigma_p(B^*, l_1^*) = \sigma(B, l_1) = \left\lbrace \lambda \in \mathbb{C} : \left| \frac{2 (r_1 - \lambda) (r_2 - \lambda)}{s_1s_2 - t_1(r_2 - \lambda)- t_2(r_1 - \lambda) + \sqrt{\chi}} \right| \leq 1 \right\rbrace,
\end{equation*}
where $\chi$ is mentioned in \eqref{chi}.
\end{theorem}

\begin{proof}
Proceeding similarly upto the equation \eqref{l1} of Theorem \ref{pointlp*}, it can be proved that, $\sigma(B, l_1) \subseteq \sigma_p(B^*, l_1^*)$. To prove the converse part we have
\[\sigma(B, l_1) \subseteq  \sigma_p(B^*, l_1^*) \subseteq \sigma(B^*, l_1^*) = \sigma(B, l_1).\]
\end{proof}

\begin{theorem}\label{res_cont_l1}
The residual spectrum and continuous spectrum of $B$ over $l_1$ are given by
\begin{enumerate}
\item[(i)] $\sigma_r(B, l_1)=\sigma(B, l_1)=\left\lbrace \lambda \in \mathbb{C} : \left| \frac{2 (r_1 - \lambda) (r_2 - \lambda)}{s_1s_2 - t_1(r_2 - \lambda)- t_2(r_1 - \lambda) + \sqrt{\chi}} \right| \leq 1 \right\rbrace$,
\item[(ii)]$\sigma_c(B, l_1)= \emptyset,$
\end{enumerate}
Where $\chi$ is mentioned in Theorem \ref{spectrum_l1}.
\end{theorem}

\begin{proof}
Proofs are similar as Theorem \ref{res_cont}.
\end{proof}

\begin{theorem}\label{goldberg_l1}
The operator $B$ satisfies the following relations
\begin{enumerate}
\item[(a)]  $A_3 \sigma(B, l_1)=B_3 \sigma(B, l_1)= C_3 \sigma(B, l_1) = \emptyset ,$
\item[(b)] $B_2 \sigma(B, l_1) = \sigma_c(B, l_1) = \emptyset,$
\item[(c)] $C_2 \sigma(B, l_1) \supseteq \left\lbrace \lambda \in \mathbb{C} : \left| \frac{2 (r_1 - \lambda) (r_2 - \lambda)}{s_1s_2 - t_1(r_2 - \lambda)- t_2(r_1 - \lambda) + \sqrt{\chi}} \right| \leq 1 \right\rbrace \setminus \{r_1, r_2\},$ 
\item[(d)] $C_1 \sigma(B, l_1) \subseteq \{r_1, r_2\}.$
\end{enumerate}
\end{theorem}

\begin{proof}
Proofs are similar as Theorem \ref{goldberg}.
\end{proof}

\begin{theorem}
The operator $B$ satisfies the following relations
\begin{enumerate}
\item[(a)] $ \sigma_{ap}(B, l_1)\supseteq \left\lbrace \lambda \in \mathbb{C} : \left| \frac{2 (r_1 - \lambda) (r_2 - \lambda)}{s_1s_2 - t_1(r_2 - \lambda)- t_2(r_1 - \lambda) + \sqrt{\chi}} \right| \leq 1 \right\rbrace\setminus \{r_1, r_2\} , $
\item[(b)] $\sigma_{ap}(B^*, l_1^*) = \left\lbrace \lambda \in \mathbb{C} : \left| \frac{2 (r_1 - \lambda) (r_2 - \lambda)}{s_1s_2 - t_1(r_2 - \lambda)- t_2(r_1 - \lambda) + \sqrt{\chi}} \right| \leq 1 \right\rbrace,$
\item[(c)] $\sigma_{\delta}(B, l_1) = \sigma_{co}(B, l_1) = \left\lbrace \lambda \in \mathbb{C} : \left| \frac{2 (r_1 - \lambda) (r_2 - \lambda)}{s_1s_2 - t_1(r_2 - \lambda)- t_2(r_1 - \lambda) + \sqrt{\chi}} \right| \leq 1 \right\rbrace.$
\end{enumerate}
\end{theorem}

\begin{proof}
Since $\sigma_{ap}(B, l_1) = \sigma(B, l_1) \setminus C_1 \sigma(B, l_1),$ the result in $(a)$ follows from Theorem \ref{goldberg_l1}. The other results follow from Proposition \ref{proposition}.
\end{proof}

\section{Example}
Here some examples are given.
\begin{enumerate}
\item[(i)] Let $B_1$ be an operator of the form \eqref{mainop3} where, $r_1 = 1, r_2 = i,$ $s_1 = 2, s_2 = 1,$ and $t_1 = -i, t_2 = 1.$ Then the spectrum of the operator $B_1$ is given by
\begin{equation*}
\sigma(B_1, l_p)=\sigma(B_1, l_1) = \left\lbrace \lambda \in \mathbb{C} : \left|\frac{2(1 - \lambda)(i - \lambda)}{(1-i)\lambda+\sqrt{2i\lambda^2+(4-4i)\lambda-4}} \right| \leq 1 \right\rbrace.
\end{equation*}

\item[(ii)] 
Let $B_2$ be an operator of the form \eqref{mainop3} where, $r_1 = i, r_2 = 2,$ $s_1 = 1+i, s_2 = 1$ and $t_1 = t_2 = 0$ then the spectrum of the operator $B_2$ is given by the set
\begin{equation*}
\sigma(B_2, l_p)=\sigma(B_2, l_1) = \left\lbrace \lambda \in \mathbb{C} : |(i - \lambda)(2 - \lambda)| \leq \sqrt{2} \right\rbrace.
\end{equation*}

\end{enumerate}

\section{Conclusion}
Many results on the spectrum and fine spectrum of difference operators obtained by earlier researchers can be derived from our results. Here we mention some of them.

\begin{enumerate}
\item[(i)] If $r_1 = r_2 = r,$ $s_1 = s_2 = s$ and $t_1 = t_2 = t$ then the operator $B$ reduces to the operator $B(r,s,t)$ 
and some of the results given in \cite{l1bv3,lpbvp3} follow from our results.

\item[(ii)] If $r_1 = r_2 = r,$ $s_1 = s_2 = s$ and $t_1 = t_2 = 0$ then the operator $B$ reduces to the operator $B(r,s).$ Therefore some of the results given in \cite{l1bv,lpbvp} follow from our results.

\item[(iii)] If $r_1 = r_2 = 1,$ $s_1 = s_2 = -1$ and $t_1 = t_2 = 0$ then the operator $B$ reduces to the difference operator $\Delta$ and the corresponding results given by \cite{deltal1bv} follows from our reults.

\item[(iv)] If $r_1 = r_2 = s,$ $s_1 = s_2 = 1-s$ and $t_1 = t_2 = 0$ where $s$ is real number and $s \neq 0, 1$ we get the Zweier matrix $Z^s$ and some of the results in \cite{zew} follows from our results.

\end{enumerate}

\end{document}